\documentclass{amsart}
\usepackage{amsmath, amssymb}
\usepackage{amstext, amscd, amsfonts}
\usepackage{mathtools}
\usepackage{cancel}
\usepackage{amsthm}
\usepackage{enumerate}
\usepackage{latexsym}
\usepackage{verbatim, epsfig}
\usepackage{dsfont}
\usepackage{ifthen}
\usepackage{xcolor}
\usepackage{amsbsy} 
\usepackage{varioref} 
\usepackage{url}
\usepackage[english]{babel}  
\usepackage{microtype} 

\DeclareMathOperator{\supp}{supp}

\newtheorem{thm}{Theorem}[section]
\newtheorem{thm*}{Theorem}
\newtheorem{theorem}{Theorem}[section]
\newtheorem{theorem*}{Theorem}

\newtheorem{corollary}[thm]{Corollary}

\newtheorem{lemma}[thm]{Lemma}

\theoremstyle{definition}

\newtheorem{defn*}{Definition}
\newtheorem{definition}{Definition}[section]
\newtheorem{definition-st}{Definition}

\newtheorem{example}{Example}[section]
\newtheorem{example-st}{Example}

\newtheorem{rem-st}{Remark}

\newtheorem{remark-st}{Remark}

\numberwithin{equation}{section}
\numberwithin{figure}{section}

\title{Multiplicity of measures under factor codes and class degree joinings}
\author{Jisang Yoo}
\address{Ajou University, Suwon, South Korea}
\email{(replace X with my last name) jisang.X.ac+rs12@gmail.com}

\begin{document}
\maketitle
\begin{abstract} Given a finite-to-one factor code $\pi: X \to Y$ between irreducible sofic shifts and an ergodic $\nu$ on $Y$ with full support, it is known that the fiber $\pi^{-1}_*(\nu)$ has at most $d_\pi$ ergodic measures in it where $d_\pi$ is the degree of $\pi$. We introduce the notion of multiplicity for ergodic measures on $X$ (that depends on $\pi$) and we prove that $d_\pi$ is the sum of the multiplicity of $\mu$ where $\mu$ runs over the ergodic measures in $\pi^{-1}_*(\nu)$. We also build an appropriate generalization to infinite-to-one factor codes in relation to class degree and relatively maximal measures. We also define the notion of degree joining (for finite-to-one factor codes) and class degree joining (for infinite-to-one factor codes) which are the main tool for establishing our results.
\end{abstract}

\section{Introduction}


It is known that any irreducible SFT has exactly one invariant measure of maximal entropy. It is also known that as soon as we move to the relative setting (where we restrict our attention to the measures in $\pi^{-1}(\nu)$ given a factor code $\pi:X\to Y$ and an ergodic measure $\nu$ on $Y$), the realm of uniqueness is lost. The simplest well known example demonstrating this is the following example:

Let $X, Y$ be full two shifts. Define $\pi$ by $\pi(x) = y$ where $y_i = x_i + x_{i+1} \pmod 2$. The factor code $\pi$ is 2-to-1. For each $0 < p < 1$, define $\mu_p$ to be the Bernoulli measure on $X$ with probability $p$ for value $1$. Let $\mu'_{p} = \mu_{1-p}$. Then $\mu$ and $\mu'$ project to a common measure $\nu = \pi(\mu) = \pi(\mu')$ on $Y$. The two measures $\mu_p, \mu'_p$ are distinct unless $p = \frac12$. The two measures are both ergodic measures of maximal relative entropy: measures which maximize entropy within $\pi^{-1}(\nu)$.

Petersen, Quas, and Shin showed a few years ago that the number of ergodic measures of maximal relative entropy (given a (possibly infinite-to-one) factor code $\pi: X\to Y$ on an irreducible SFT and an ergodic measure $\nu$ on $Y$) is always finite \cite{PQS-MaxRelEnt}. Allahbakhshi and Quas then improved the result and obtained a conjugacy-invariant upper bound and named it class degree \cite{all2013classdegrelmaxent}. Class degree is roughly the finite number of equivalence classes, called transition classes, in the fiber $\pi^{-1}(y)$ where $y$ is some typical point in $Y$. They asked if the similar result can be established for the number of ergodic measures of maximal pressure of f, i.e., ones maximizing $h(\mu)+ \int f d\mu$ within $\pi^{-1}(\nu)$, where $f$ is some regular enough potential function on $X$. We will call this \emph{Question A}. Recently a positive answer was given for potential functions with summable variation \cite{yoo2014releqclass} and the effort to extend to other potential functions is in plan by the author, Allahbakhshi and Antonioli.

Aside from thermodynamic consideration, the notion of class degree has been shown to be interesting on its own as Allahbakhshi, Hong and Jung showed that it reveals the topological structure of infinite-to-one factor codes \cite{all2014structureclass} and they asked if each transition class over a $\nu$-generic point contains some generic point for an ergodic measure of maximal relative entropy. This question, which we will call Question B, can be answered negatively with a quick example, so we modify the question by weakening it in two says:
\begin{enumerate}
\item replace $\nu$-generic points with $\nu$-a.e. points,
\item and replace maximization within $\pi^{-1}(\nu)$ (measures of maximal relative entropy) with maximization within an equivalence class in $\pi^{-1}(\nu)$ (class maximal measures)
\end{enumerate}
and ask if each transition class can be associated with at least one class maximal ergodic measure. To give a precise formulation of this modified question, which we will call \emph{Question B'}, and answering it affirmatively and also answering its obvious generalization to maximization of pressure of $f$ is one of the our main results.

Question B' is dual to Question A', a generalization of Question A to class maximal measures. We note that the proof of the positive answer to Question A for $f$ with summable variation relies on a lemma which essentially proves the positive answer to the stronger Question A' for such $f$. Therefore, combining answers to Question A' and Question B' for $f=0$ in particular, the realm of uniqueness is restored by moving further to the setting of restricting our attention to an equivalence class in $\pi^{-1}(\nu)$. In order to establish this setting, we develop a way to join class maximal measures together called class degree joining and show that they can be joined together in a unique way and introduce the notion of class multiplicity.

We first establish specialized results for finite-to-one factor codes. The specialized results for finite-to-one factor codes enable us to prove, in some concrete example of a 5-to-1 factor code and a broad class of $\nu$, that the number of ergodic measures in $\pi^{-1}(\nu)$ is 3. Previously there has been no tools to establish examples of finite-to-one factor codes such that one can exhibit a class of measures $\nu$ on $Y$ with the property that the number of its ergodic lifts is strictly between 1 and the degree of the factor code. We also demonstrate that, given a finite-to-one factor code and an ergodic measure $\mu$ on $X$, as soon as one knows a concrete way to list all points in $\pi^{-1}(\pi(x))$ from $x$, one also has a way to list all ergodic measures in $\pi^{-1}(\pi(\mu))$ and count the number of them. Using this, we prove a nontrivial result that in an appropriate generalization of the mod 2 factor code example, the number of ergodic measures in $\pi^{-1}(\pi(\mu))$ always divides the degree of the factor code.

The problem of counting the number of ergodic measures of maximal relative entropy under an infinite-to-one factor code on an SFT and a \emph{fully supported} ergodic measure $\nu$ on $Y$ reduces to the problem of counting the same under a \emph{finite-to-one} factor code on a \emph{sofic} shift space \cite{JungYoo2015decomp}. In order to complement it, our results concern mainly the following two cases:
\begin{enumerate}
\item a finite-to-one factor code $\pi: X \to Y$ on a sofic shift space $X$ and a fully supported ergodic $\nu$ on $Y$
\item a possibly infinite-to-one factor code $\pi: X \to Y$ on an SFT $X$ and an ergodic $\nu$ on $Y$ which may not have full support
\end{enumerate}

By allowing $\nu$ to not have full support, we hope that some further structure of infinite-to-one factor codes can be studied by using measures on $Y$ that are between two opposite examples: fully supported measures and periodic orbits.

Our results reveal that there are only two causes for the number of ergodic measures of maximal relative entropy to be strictly smaller than the class degree:
\begin{enumerate}
\item A class maximal measure having a different entropy than another class maximal measure (of the same pushforward image)
\item A class multiplicity of a class maximal measure being greater than one
\end{enumerate}



\section{Background}

Unless stated otherwise, a topological dynamical system here means a self homeomorphism $T: X \to X$ of a compact metric space $X$, and a shift of finite type (SFT) means a two-sided shift of finite type with finite alphabet, and a measure on $X$ means a Borel probability measure on it. A \emph{factor code} is a factor map between two shift spaces. When we say $\pi: X\to Y$ is a factor code on an SFT $X$, it is therefore assumed that $Y$ is the image of $X$ under $\pi$ (and hence a sofic shift space).
Given a factor map $\pi: (X, T) \to (Y,S)$ between topological dynamical systems, we will say a measure $\mu$ on $X$ is a \emph{lift} or a \emph{preimage} of a measure $\nu$ on $Y$ if the pushforward image of $\mu$ under $\pi$ is $\nu$, i.e., if $\pi\mu = \nu$.

For the section on finite-to-one factor codes, we will use the following facts.
\begin{theorem}[Theorem  8.1.19 in~\cite{LM}]\label{thm:fto-conditions}
  Let $X$ be an irreducible sofic shift and $\pi: X\to Y$ a factor code. Then the following are equivalent.
  \begin{enumerate}
  \item For every $y \in Y$, the fiber $\pi^{-1}(y)$ is countable.
  \item For every $y \in Y$, the fiber $\pi^{-1}(y)$ is finite.
  \item There is $M \in \mathbb N$ such that, for every $y \in Y$, $|\pi^{-1}(y)| \le M$.
  \item $h(X) = h(Y)$.
  \end{enumerate}
\end{theorem}

\begin{theorem}[Lemma 9.1.13 in~\cite{LM}]\label{thm:fto-and-transitive}
  Let $X$ be an irreducible sofic shift and $\pi: X\to Y$ a finite-to-one factor code. Then a point $x \in X$ is doubly transitive if and only if $\pi(x)$ is.
\end{theorem}

\begin{theorem}[Corollary  9.1.14 in \cite{LM}]\label{thm:fto-has-degree}
  Let $X$ be an irreducible sofic shift and $\pi: X\to Y$ a finite-to-one factor code. There is $d_\pi \in \mathbb N$ such that each doubly transitive point in $Y$ has exactly $d_\pi$ pre-images. This number $d_\pi$ is called the \emph{degree} of the factor code $\pi$.
\end{theorem}

A set of points in a 1-step SFT is \emph{mutually separated} if they never occupy the same symbol at the same time. Mutual separatedness for a set of words of length $m$ in a 1-step SFT is defined similarly.

\begin{theorem}[Variation of Proposition 9.1.9 in \cite{LM}]\label{thm:fto-mut-sep}
  Let $X$ be an irreducible 1-step SFT and $\pi: X\to Y$ a finite-to-one 1-block factor code. Let $y \in Y$. Then there are at least $d_\pi$ mutually separated points in the fiber $\pi^{-1}(y)$.
\end{theorem}


\section{Measures with full support}
\begin{lemma}\label{lem:fto-full}
  Let $X$ be an irreducible sofic shift and $\pi: X\to Y$ a finite-to-one factor code. An invariant measure $\mu$ on $X$ has full support if and only if the pushforward image $\pi\mu$ has full support.
\end{lemma}
\begin{proof}
  The nontrivial direction is the if direction. Suppose $\pi\mu$ has full support but $\mu$ does not. Let $X_0 = \supp(\mu)$. Then $X_0$ is a proper subshift of $X$, hence $h(X_0) < h(X)$, but we also have $\pi(X_0) = \supp(\pi\mu) = Y$ and hence $h(X_0) \ge h(Y)$. Therefore, $h(X) > h(Y)$ which contradicts the assumption that $\pi: X\to Y$ is finite-to-one.
\end{proof}

\section{Relative joinings}
Given a factor map $\pi: (X,T) \to (Y,S)$ between topological dynamical systems, an \emph{$n$-fold $\pi$-relative joining} is an invariant measure $\lambda$ on $X^n$ for which the subset $X^n_\pi := \{(x^{(1)}, x^{(2)}, \dots, x^{(n)}) \in X^n : \pi(x^{(1)}) = \pi(x^{(2)}) = \dots = \pi(x^{(n)})\}$ is a full measure set. We will say that such a measure $\lambda$ is a relative joining of margins $\mu_1, \dots, \mu_n$ over image $\nu$ if $p_i\lambda = \mu_i$ for each $i$, where $p_i: X^n \to X$ is the projection to the $i$-th, and $\pi p_i \lambda = \nu$ for some $i$ (and hence for all $i$). We will say such a measure $\lambda$ is \emph{separating} if for $\lambda$-a.e. $(x^{(1)}, x^{(2)}, \dots, x^{(n)})$, the points $x^{(1)}, x^{(2)}, \dots, x^{(n)}$ are $n$ distinct points.

In the following theorem, we define and prove the existence of a degree joining, which is a particular way of joining together all ergodic pre-images of a fully supported ergodic measure on $Y$. This proof is redundant since its generalization to arbitrary finite-to-one factor maps (between general dynamical systems) is proved in one of the next sections, but it is this proof that generalizes to the similar result for infinite-to-one factor codes later (between symbolic systems).
\begin{theorem}\label{thm:degree-joining-exists}
  Let $X$ be an irreducible sofic shift and $\pi: X\to Y$ a finite-to-one factor code. Let $\nu$ be a fully supported ergodic measure on $Y$. Then there exists an ergodic $d_\pi$-fold separating relative joining over $\nu$. We will call such a joining a \emph{degree joining} over $\nu$ with respect to $\pi$.
\end{theorem}
\begin{proof}
Denote $d = d_\pi$.
We prove first for the case when $X$ is an irreducible SFT. We may assume that $\pi$ is 1-block and $X$ is 1-step.
Let $Z$ be the set of all $(x^{(1)}, x^{(2)}, \dots, x^{(d)}) \in X^d$ such that $\pi(x^{(1)}) = \pi(x^{(2)}) = \dots = \pi(x^{(d)})$ and that the points $x^{(1)}, x^{(2)}, \dots, x^{(d)}$ are $d$ distinct mutually separated points. Then $Z$ is a shift space. Theorem~\ref{thm:fto-mut-sep} implies that the obvious sliding block code $Z \to Y$ is a factor code. Therefore we can lift the ergodic $\nu$ on $Y$ to an ergodic measure $\lambda$ on $Z$ via the factor code $Z \to Y$. It is easy to check that this measure $\lambda$ is a degree joining over $\nu$.

Now suppose $X$ is strictly sofic.
Let $\pi_R: X_R \to X$ be the minimal right resolving presentation of $X$. We have $d = d_\pi = d_{\pi\circ\pi_R}$.
We may assume that $\pi$ and $\pi_R$ are 1-block codes and $X_R$ is a 1-step SFT.
Let $\lambda_R$ be a degree joining over $\nu$ with respect to $\pi\circ\pi_R$. Let $\lambda$ be the pushforward of $\lambda_R$ under the obvious sliding block code $X_R^d \to X^d$. It is easy to check that $\lambda$ is a $d$-fold relative joining over $\nu$.
To see that $\lambda$ is separating, first notice that $\nu$-a.e. $y$ is doubly transitive, because $\nu$ has full support. Hence, for $\lambda_R$-a.e. $(x^{(1)}, x^{(2)}, \dots, x^{(d)}) \in X_R^d$, each $x^{(i)}$ is doubly transitive by Theorem~\ref{thm:fto-and-transitive}, but $\pi_R$ must be injective on doubly transitive points because $\pi_R$ has degree 1, and therefore images of $x^{(1)}, x^{(2)}, \dots, x^{(d)}$ under $\pi_R$ are $d$ distinct points. Therefore $\lambda$ is separating.
\end{proof}

\begin{lemma}\label{lem:ergodic-rel-join-exists}
  Let $(X, T, \mu)$ and $(Y, S, \nu)$ be two ergodic measure preserving systems with $(Z, R, \rho)$ as its common factor. Then there is an ergodic relative joining of the two systems over the common factor, i.e., there is a measure on $X \times Y$ which is an ergodic joining of $\mu$ and $\nu$ and is a relative joining over $\rho$.
\end{lemma}
\begin{proof}
  Let $\mu \otimes_\rho \nu$ be the relatively independent joining of the two measures over $\rho$. Since $\mu, \nu, \rho$ are ergodic, almost all measures in the ergodic decomposition of $\mu \otimes_\rho \nu$ must also have $\mu, \nu$ as their margins and $\rho$ as their image in $Z$. It is also easy to check that almost all measures in the ergodic decomposition is a relative joining.
\end{proof}

Degree joinings are universal with respect to other relative joinings over the same image in the following sense.
\begin{theorem}\label{thm:degree-joining-universal}
  Let $X$ be an irreducible sofic shift and $\pi: X\to Y$ a finite-to-one factor code. Let $\nu$ be a fully supported ergodic measure on $Y$. Let $\lambda$ be a degree joining over $\nu$ and $\lambda'$ an $n$-fold ergodic relative joining over $\nu$. Then there is a function $f: \{1,\dots, n\} \to \{1,\dots,d_\pi\}$ such that $\lambda' = p_f \lambda$ where $p_f: X^{d_\pi} \to X^n$ is the obvious map induced by $f$.
\end{theorem}
\begin{proof}
  There is an ergodic relative joining, say $\lambda''$, of $\lambda$ and $\lambda'$ over $\nu$ in the sense of the previous lemma. $\lambda''$ is a measure on $X^{d} \times X^{n}$. For $\lambda''$-a.e. $(x_1, \dots, x_d, x'_1, \dots, x'_n)$, we have that the points $x_1, \dots, x_d$ are $d$ distinct pre-images of a doubly transitive point in $Y$ (hence they are all pre-images of that point in $Y$) and that $x'_1,\dots,x'_n$ are pre-images of the same point in $Y$, and therefore in particular, the point $x'_1$ is equal to one and only point among $x_1, \dots, x_d$. Therefore there is a measurable function $g: X^{d} \times X^{n} \to \{1,\dots, d\}$ such that $$x'_1 = x_{g(x_1, \dots, x_d, x'_1, \dots, x'_n)}$$ holds for $\lambda''$-almost all $(x_1, \dots, x_d, x'_1, \dots, x'_n)$.

Since $g$ is invariant and $\lambda''$ is ergodic, $g$ must be constant $\lambda''$-a.e.. Define $f(1)$ to be this constant. Define $f(2), \dots, f(n)$ similarly. The function $f: \{1,\dots, n\} \to \{1,\dots,d\}$ defined in this way has the desired property because 
$$(x'_1,\dots, x'_n) = (x_{f(1)}, \dots, x_{f(n)}) = p_f(x_1,\dots,x_d)$$
holds for $\lambda''$-a.e. $(x_1, \dots, x_d, x'_1, \dots, x'_n)$.
\end{proof}
Since any relative joining over $\nu$ decomposes by ergodic decomposition into ergodic relative joinings over $\nu$, we have just classified all possible relative joinings over $\nu$.

Universality implies uniqueness of degree joining up to permutation as proved in the following theorem.
\begin{theorem}\label{thm:degree-joining-unique}
  Let $X$ be an irreducible sofic shift and $\pi: X\to Y$ a finite-to-one factor code. Let $\nu$ be a fully supported ergodic measure on $Y$. If $\lambda$ and $\lambda'$ are degree joinings over $\nu$, then there is a permutation $f$ of $\{1, \dots, d_\pi\}$ such that $\lambda' = p_f\lambda$.
\end{theorem}
\begin{proof}
  There is a function $f: \{1,\dots, d\} \to \{1,\dots,d\}$ such that $\lambda' = p_f \lambda$.
  Suppose that $f$ is not surjective. Without loss of generality, we may assume $f(1) = f(2) = 1$.

  For $\lambda$-a.e. $(x_1, \dots, x_d)$ we have that $p_f(x_1, \dots, x_d)$ is of the form $(x'_1, \dots, x'_d)$ with $x'_1 = x'_2$. Therefore, for $\lambda'$-a.e. $(x'_1, \dots, x'_d)$, we have $x'_1 = x'_2$ but this contradicts the assumption that $\lambda'$ is separating.
\end{proof}

Having established the uniqueness of degree joining, we now show that its margins are precisely the ergodic lifts of $\nu$.
\begin{theorem}
  Let $X$ be an irreducible sofic shift and $\pi: X\to Y$ a finite-to-one factor code. Let $\nu$ be a fully supported ergodic measure on $Y$ and $\lambda$ a degree joining over it. Then $$\{p_i \lambda: 1\le i \le d_\pi\}$$ is the set of all ergodic measures in $\pi^{-1}\nu$.
\end{theorem}
\begin{proof}
  Each margin $p_i\lambda$ is an ergodic measure on $X$ that maps to $\nu$ because $\lambda$ is an ergodic joining over $\nu$.

  Each ergodic measure in $\pi^{-1}\nu$ is a 1-fold ergodic relative joining over $\nu$ and hence Theorem~\ref{thm:degree-joining-unique} applies to it and therefore is a margin of $\lambda$.
\end{proof}

\section{Multiplicity}

\begin{definition}
  Let $X$ be an irreducible sofic shift and $\pi: X\to Y$ a finite-to-one factor code. Let $\mu$ be a fully supported ergodic measure on $X$. The \emph{multiplicity}, denoted $m_\pi(\mu)$, of $\mu$ with respect to $\pi$ is the number of times it appears as a margin in a degree joining over $\pi\mu$. In other words,
  $$m_\pi(\mu) := \#\{i : 1\le i \le d_\pi,\ p_i \lambda = \mu\}$$
  where $\lambda$ is a degree joining over $\pi\mu$.
\end{definition}
Since degree joining is unique up to permutation, the notion of multiplicity above is well defined. Our original purpose in defining this notion is to establish the following result which is trivial from the way multiplicity is defined.
\begin{theorem}
  Let $X$ be an irreducible sofic shift and $\pi: X\to Y$ a finite-to-one factor code. Let $\nu$ be a fully supported ergodic measure on $Y$. Then
  $$d_\pi = \sum_\mu m_\pi(\mu)$$
  where $\mu$ runs over all ergodic lifts of $\nu$.
\end{theorem}

In particular, the degree of $\pi$ is the upper bound on the number of ergodic lifts, which is already known. The new result is that their multiplicities sum to the degree.

\begin{theorem}
  Let $X$ be an irreducible sofic shift and $\pi: X\to Y$ a finite-to-one factor code. Let $\nu$ be a fully supported ergodic measure on $Y$. Then for $\nu$-a.e. $y \in Y$, each point in the fiber $\pi^{-1}(y)$ is a generic point for some ergodic measure in $\pi^{-1}(\nu)$. Furthermore, let $\mu_1, \dots, \mu_k$ be all ergodic lifts of $\nu$ and let $m_1,\dots, m_k$ be their multiplicities. Then for $\nu$-a.e. $y \in Y$, the fiber $\pi^{-1}(y)$ consists of $m_1$ points generic for $\mu_1$, and $m_2$ points generic for $\mu_2$, \dots, and $m_k$ points generic for $\mu_k$.
\end{theorem}
\begin{proof}
  Let $\lambda$ be a degree joining over $\nu$. For $\lambda$-a.e. $(x_1, \dots, x_d)$, we have that $x_1$ is generic for $p_1 \lambda$, and $x_2$ is generic for $p_2 \lambda$, and so on. The desired conclusion follows by transferring to $Y$.
\end{proof}

\begin{theorem}
  Let $X$ be an irreducible sofic shift and $\pi: X\to Y$ a finite-to-one factor code. Let $\mu$ be a fully supported ergodic measure on $X$ and let $m$ be its multiplicity. Let $\{\mu_y\}_{y \in Y}$ be the disintegration of $\mu$ over $Y$. Then
  \begin{enumerate}
  \item For $\pi\mu$-a.e. $y$, the measure $\mu_y$ is uniformly distributed on $G_\mu \cap \pi^{-1}(y)$, where $G_\mu$ is the set of points generic for $\mu$, and there are $m$ points in $G_\mu \cap \pi^{-1}(y)$.
  \item $m$ is the maximum number such that there is an $m$-fold separating relative joining of $\mu, \dots, \mu$.
  \item $(\mu \otimes_{\pi\mu} \mu) \{(x, x'): x=x'\} = \frac1m$
  \end{enumerate}
\end{theorem}
\begin{proof}
(1)
The previous theorem ensures that there is a Borel subset $Y_0 \subset Y$ such that $\pi\mu(Y_0)=1$ and that the size of $G_\mu \cap \pi^{-1}(y)$ is $m$ for each $y \in Y_0$.
For each $y \in Y_0$, let $M_y \subset X^m$ be the set of all $m!$ orderings of the $m$ distinct points in $G_\mu \cap \pi^{-1}(y)$.

We can obtain an $m$-fold separating relative joining $\lambda$ of $\mu, \dots, \mu$ by projecting a degree joining over $\pi\mu$ to the appropriate $m$ coordinates.
Then for $\lambda$-a.e. $(x_1,\dots,x_m)$, the point $\pi\circ p_1(x_1)$ is in $Y_0$ and the sequence $(x_1,\dots,x_m)$ is an element of $M_{\pi\circ p_1(x_1)}$.

Consider the disintegration of $\lambda$ via $\pi\circ p_1$ over $Y$ and write $\lambda = \int_Y \lambda_y d\pi\mu(y)$
Then for $\pi\mu$-a.e. $y$ we have
\begin{equation}
  \label{eq:My}
  \lambda_y(M_y) = 1
\end{equation}
and so $\lambda_y$ is an atomic measure supported on $M_y$.

Consider the map $F: (X^m, \lambda) \otimes (\{1,2,\dots,m\}, U_m) \to X$, where $U_m$ is the uniform probability distribution on $m$ digits, defined by $F(x_1,\dots, x_m, i) = x_i$. Then we have $$F(\lambda \otimes U_m) = \frac{\mu+ \dots + \mu}{m} = \mu$$
By \eqref{eq:My}, it follows that for $\pi\mu$-a.e. $y \in Y$, the measure $F(\lambda_y \otimes U_m)$ is the uniform probability distribution on the $m$ points in $G_\mu \cap \pi^{-1}(y)$.
So it is sufficient to show that $y \mapsto F(\lambda_y \otimes U_m)$ is another disintegration of $\mu$ via $\pi$.
Indeed, we have
\begin{align*}
  \mu &= F(\lambda \otimes U_m)\\
  &= \int_Y F(\lambda_y \otimes U_m) d\pi\mu(y)
\end{align*}
and we have already shown that $F(\lambda_y \otimes U_m)$ is supported on $\pi^{-1}(y)$.

(2) We have already shown the existence of an $m$-fold joining with the specified property. Suppose $\lambda'$ is an $m+1$-fold such joining. Then for $\lambda'$-a.e. $(x_1,\dots, x_{m+1})$, the points $x_1,\dots, x_{m+1}$ are $m+1$ distinct points and they are all in $G_\mu \cap \pi^{-1}(\pi\circ p_1 (x_1))$. Therefore, for $\pi\mu$-a.e. $y$, the size of $G_\mu \cap \pi^{-1}(y)$ is at least $m+1$. This contradicts the previous theorem and so there can be no such $m+1$-fold joining $\lambda'$.

(3) For $\pi\mu$-a.e. $y$, we have $(\mu_y \otimes \mu_y) \{(x, x'): x=x'\} = \frac1m$ because $\mu_y$ is the uniform distribution on $m$ points. Integrating over $Y$ gives the desired result.

\end{proof}

Each of the three properties shown in the above theorem can be taken to be an alternative characterization of the notion of multiplicity.

\section{Multiplicity and degree under factor maps between general systems}

In this section, we show that the results in the previous two sections generalize to arbitrary factor maps between topological dynamical systems as long as we have the condition that $\pi^{-1}(y)$ is a finite set for $\nu$-a.e. $y \in Y$ where $\nu$ is an ergodic measure on $Y$ which may or may not have full support. First, we define the notion of degree for an arbitrary ergodic measure on $Y$. For that, we need the following lemma.

\begin{lemma}
  Let $(X,T)$ and $(Y,S)$ be topological dynamical systems and $\pi: X\to Y$ a factor map. Then the map $F: Y \to \{1,2,\dots\} \cup \{\infty\}$ defined by $y \mapsto |\pi^{-1}(y)|$ is measurable and is constant a.e. with respect to each ergodic measure on $Y$.
\end{lemma}
\begin{proof}
  We do not know if $F$ is Borel-measurable, but we can show that it is universally measurable. For each $k \in \mathbb N$, the set $ \{y\in Y: |\pi^{-1}y| \ge k\}$ is the projection of a Borel subset in $X^k \times Y$, namely, the subset consisting of all $(x_1,x_2,\dots,x_k,y) \in X^k \times Y$ for which $\pi(x_i)=y$ for all $1\le i \le k$ and $x_i \neq x_j$ for all $1\le i < j \le k$, and therefore this set is an analytic subset of $Y$, and hence a universally measurable set. It follows that the map $F$ is universally measurable.
  Since the map $F$ is invariant with respect to the action $S$, it must be constant a.e. with respect to each ergodic measure on $Y$.
\end{proof}

For each ergodic $\nu$ on $Y$, we define the \emph{degree} of $\nu$ relative to $\pi$ to be the number $d \in \{1,2,\dots\} \cup \{\infty\}$ such that for $\nu$-a.e. $y \in Y$, there are precisely $d$ points in the fiber $\pi^{-1}(y)$. We will denote this number by $d_{\pi,\nu}$, and if $\pi$ is understood, by $d_\nu$.

If a factor map $\pi: X \to Y$ is such that $d_{\pi,\nu} = d_{\pi,\nu'}$ whenever $\nu$ and $\nu'$ are fully supported ergodic measures on $Y$, it makes sense to define the \emph{degree of the factor map} to be $d_{\pi,\nu}$ and denote it by $d_\pi$. If $\pi: X \to Y$ is a finite-to-one factor code on an irreducible sofic shift, then its degree defined in this way is equivalent to the degree defined by using doubly transitive points on $Y$, since the set of doubly transitive points has full measure with respect to each fully supported ergodic $\nu$ on $Y$.


To establish the existence of a degree joining for the general case, we need another measurability lemma:

\begin{lemma}
  Let $\pi: X \to Y$ be a Borel-measurable map between Polish spaces. Let $A\subset Y$ be a Borel subset. Then the map $F_A: Y \to \{0,1,2,\dots\} \cup \{\infty\}$ defined by $y \mapsto |\pi^{-1}(y) \cap A|$ is universally measurable.
\end{lemma}
\begin{proof}
  For each $k \in \mathbb N$, the set $ \{y\in Y: |\pi^{-1}y \cap A| \ge k\}$ is the projection of a Borel subset in $X^k \times Y$, namely, the subset consisting of all $(x_1,x_2,\dots,x_k,y) \in X^k \times Y$ for which $\pi(x_i)=y$ and $x_i \in A$ for all $1\le i \le k$ and $x_i \neq x_j$ for all $1\le i < j \le k$, and therefore this set is an analytic subset of $Y$, and hence a universally measurable set. It follows that the map $F_A$ is universally measurable.
\end{proof}

\begin{theorem}
  Let $(X,T)$ and $(Y,S)$ be topological dynamical systems and $\pi: X\to Y$ a factor map. Let $\nu$ be an ergodic measure on $Y$ with finite degree $d := d_{\pi,\nu} < \infty$.
  Then there is an invariant measure $\mu$ on $X$ such that $\pi\mu=\nu$ and that $\mu_y$ (its disintegration over Y) is a uniform distribution on $\pi^{-1}(y)$ for $\nu$-a.e. $y \in Y$. This measure is unique and we will call it the \emph{canonical lift} of $\nu$ and denote it by $\ell_\pi(\nu)$.
\end{theorem}
\begin{proof}
  (Existence)
  For each Borel $A\subset Y$, we define $$\mu(A) = \frac{\int_Y F_A(y) d\nu(y)}{d}$$
  This is well defined because of the previous lemma and it is easy to verify that $\mu$ is countably additive and $\mu(\emptyset)=0$ and $\mu(X)=1$.

  $\mu$ is $T$-invariant because
  \begin{align*}
    \mu(T^{-1}A) &= \frac{\int_Y F_{T^{-1}A}(y) d\nu(y)}{d}\\
    &= \frac{\int_Y F_{A}(Sy) d\nu(y)}{d}\\
    &= \frac{\int_Y F_{A}(y) d\nu(y)}{d}
  \end{align*}
  where the last equality holds because $\nu$ is $S$-invariant. It is also easy to verify $\pi\mu = \nu$.

  Let $Y_0$ be a Borel subset of $Y$ such that $\nu(Y_0)=1$ and $F_X(y) = d$ for all $y \in Y_0$. Then the map $U: Y_0 \to M(X)$ defined by requiring that $U_y$ be the uniform distribution on the $d$ points in $\pi^{-1}(y)$ is a measurable map by the previous lemma again. This map $U$ is a disintegration of $\mu$ over Y, since
  \begin{align*}
    \mu(A) &= \frac{\int_Y F_A(y) d\nu(y)}{d}\\
    &=\int_{Y_0} U_y(A) d\nu(y)
  \end{align*}

  (Uniqueness)
  If $\mu'$ is another such measure, then
  \begin{align*}
    \mu' &= \int_Y U_y d\nu(y)\\
    &= \mu
  \end{align*}
\end{proof}

\begin{theorem}
  Let $(X,T)$ and $(Y,S)$ be topological dynamical systems and $\pi: X\to Y$ a factor map. Let $\nu$ be an ergodic measure on $Y$ with finite degree $d := d_{\pi,\nu} < \infty$. Then there exists an ergodic $d$-fold separating relative joining over $\nu$. We will call such a joining a \emph{degree joining} over $\nu$ with respect to $\pi$.
\end{theorem}
\begin{proof}
  Let $\mu := \ell_\pi(\nu)$ be the canonical lift of $\nu$.
  Let $\lambda$ be the $d$-fold relatively independent joining of $\mu$ over $\nu$. In other words,
  $$\lambda_y = \mu_y \otimes \mu_y \otimes \dots \otimes \mu_y$$

  Since $\mu_y$ is a uniform distribution on $d$ points, we have
  \begin{align*}
    \lambda_y(Z) &= \frac{d-1}{d} \cdot \frac{d-2}{d} \cdots \frac{1}{d}\\
    &> 0
  \end{align*}
  where $Z$ is the set of all $(x_1,\dots,x_d) \in X^d$ such that $x_i \neq x_j$ for all $1 \le i < j \le d$. In particular, we have $\lambda(Z)>0$.

  Let $\lambda = \int \lambda' d\rho(\lambda')$ be the ergodic decomposition of $\lambda$. Then, since
  $$ 0 < \lambda(Z) = \int \lambda'(Z) d\rho(\lambda')$$
  we have that for each $\lambda'$ in some $\rho$-positive set, $\lambda'(Z) > 0$, but since $\lambda'$ is ergodic and $Z$ is invariant, we get $\lambda'(Z) = 1$.
\end{proof}

All theorems in the previous two sections hold for the general case of factor maps between topological dynamical systems, as long as $\nu$ is ergodic and $d_{\pi,\nu}$ is finite.

\section{Canonical lift}

Using the notion of canonical lift, we can obtain yet another characterization of the notion of multiplicity, as weights in the ergodic decomposition of the canonical lift.
\begin{theorem}
 Let $(X,T)$ and $(Y,S)$ be topological dynamical systems and $\pi: X\to Y$ a factor map. Let $\nu$ be an ergodic measure on $Y$ with finite degree $d := d_{\pi,\nu} < \infty$. Let $\mu_1, \dots, \mu_k$ be all ergodic lifts of $\nu$ and let $m_1,\dots, m_k$ be their multiplicities. Then the ergodic decomposition of the canonical lift of $\nu$ is given by
 $$\ell_\pi(\nu) = \sum_{i=1}^{k} \frac{m_i}{d} \cdot \mu_i$$
\end{theorem}
\begin{proof}
 Let $\lambda$ be a degree joining over $\nu$. Consider the map $$F: (X^d, \lambda) \otimes (\{1,2,\dots,d\}, U_d) \to X$$ where $U_d$ is the uniform probability distribution on $d$ digits, defined by $$F(x_1,\dots, x_d, i) = x_i$$ Then we have $$F(\lambda \otimes U_d) = \frac{p_1\lambda + \dots + p_d\lambda}{d} = \sum_{i=1}^{k} \frac{m_i}{d} \cdot \mu_i$$

 It remains to show that $F(\lambda \otimes U_d)$ is the canonical lift of $\nu$. We already know that it is an invariant lift of $\nu$. For $\nu$-a.e. $y \in Y$, $\lambda_y$ is an atomic measure supported on $M_y \subset X^d$ where $M_y$ is the set of all $d!$ orderings of the $d$ distinct points in $\pi^{-1}(y)$. It is then easy to see that $F(\lambda_y \otimes U_d)$ is $U_y$, the uniform distribution on the $d$ points in $\pi^{-1}(y)$. From this, and by uniqueness of the canonical lift, we only need to show that $y \mapsto F(\lambda_y \otimes U_d)$ is a disintegration of $F(\lambda \otimes U_d)$ over $\pi$, but we already have
\begin{align*}
  F(\lambda \otimes U_d) = \int_Y F(\lambda_y \otimes U_d) d\nu(y)
\end{align*}
and we have already shown that $F(\lambda_y \otimes U_d)$ is supported on $\pi^{-1}(y)$.
\end{proof}

In the above sense, the canonical lift contains all possible ergodic lifts of $\nu$. In the above proof, we have shown that the canonical lift can be obtained from a degree joining. Similarly, a degree joining can be obtained from the canonical lift and that is indeed how the existence of a degree joining was proved.

\begin{corollary}
  Let $(X,T)$ and $(Y,S)$ be topological dynamical systems and $\pi: X\to Y$ a factor map. Let $\nu$ be an ergodic measure on $Y$ with finite degree $d := d_{\pi,\nu} < \infty$. The canonical lift $\ell_\pi(\nu)$ is ergodic if and only if there is only one (invariant) pre-image of $\nu$, in which case the canonical lift is the unique (invariant) pre-image of $\nu$.
\end{corollary}


\section{Examples}


\begin{example}
Let $N \in \mathbb N$.  Let $X = Y$ be the full $N$ shift. Then the factor code $\pi:X \to Y$ defined by
$$x = (x_i)_i \mapsto (x_{i+1}-x_{i})_i \pmod{N}$$
is a $N$-to-1 map. Indeed, if the map $s: X\to X$ is defined by
$$x = (x_i)_i \mapsto (x_i+1)_i \pmod{N}$$
then we have
$$\pi^{-1}\pi x = \{x, s(x),\dots, s^{N-1}(x) \} = \{s^k(x) : k \in \mathbb Z \}$$
for all $x \in X$. For any ergodic $\mu$ on $X$, its image $\lambda$ under the map
$$x \mapsto (x, s(x),\dots, s^{N-1}(x))$$
is a degree joining over $\pi\mu$. ($\lambda$ is ergodic because it is an image of $\mu$ under a shift-commuting map.)
Therefore, in particular, all ergodic lifts of $\pi\mu$ are in the list $\mu, s(\mu), \dots, s^{N-1}(\mu)$ and the multiplicity of $\mu$ is the number of times it appears in the list and is therefore always a divisor of $N$. The number of ergodic lifts of $\pi\mu$ also divides $N$ and $N$ is the product of that number and the multiplicity of $\mu$.

If $\mu$ is the Bernoulli measure on $X$ given by a probability vector $(\alpha_1,\dots, \alpha_N)$, then its multiplicity is $\frac{N}{L}$ where $L$ is the least period of the sequence $(\alpha_1,\dots, \alpha_N)$.
\end{example}

\begin{lemma}
  Let  $(X, T, \mu)$ be an ergodic system. Denote by $(2, S, \nu)$ the unique ergodic system consisting of two atoms. Then the following are equivalent.
  \begin{enumerate}
  \item The system $(2,S,\nu)$ is a factor of $(X,T,\mu)$.
  \item The product $\mu \otimes \nu$ is not ergodic.
  \end{enumerate}
  If these conditions hold, we will say that 2 is a factor of $\mu$.
\end{lemma}
\begin{proof}
  If $(2, S, \nu)$ is a factor of $(X, T ,\mu)$, then $(X, T, \mu) \times (2, S, \nu)$ has a factor $(2, S, \nu) \times (2, S, \nu)$ which is not ergodic and therefore the product system is not ergodic.

  It remains to show $\neg (1) \implies \neg (2)$. Let $f: X \times 2 \to \mathbb R$ be a $\mu \otimes \nu$-a.e. $T\times S$-invariant measurable function. We want to show that this function is a.e. constant. Since $f$ is invariant, $f(x, 0) = f(Tx, 1)$ and $f(x, 1) = f(Tx, 0)$ hold for a.e. $x$. So $f(x, 0) + f(x, 1)$ is $T$-invariant and hence, by the ergodicity of $T$, a.e. constant. So for some $r \in \mathbb R$ we have $f(x, 0) + f(x, 1) = r$ a.e.
On the other hand, we have $f(x, 0) - f(x, 1) = - (f(Tx, 0) - f(Tx, 1))$. So $f(x, 0) - f(x, 1)$ is a.e. zero, because otherwise it would be a.e. nonzero by the ergodicity of $T$ and then the sign of $f(x, 0) - f(x, 1)$ can be used to form a factor map to $(2, S, \nu)$ which would contradict our starting assumption. So $f(x, 0) = f(x, 1)$ holds a.e. and therefore $f(x, 0) = f(x, 1) = \frac{r}2$.
\end{proof}
\begin{example}
  Let $X=Y$ be the full 5 shift. Then the factor code $\pi:X \to Y$ defined by
$$x = (x_i)_i \mapsto (x_{i+1}+x_{i})_i \pmod{5}$$
is a 5-to-1 map. Note that unlike the previous example, we are taking the sum of two consecutive numbers instead of taking the difference. Let $\mu$ be an ergodic measure on $X$ such that 2 is not its factor. To form a degree joining over $\pi\mu$, we need some auxiliary measure. Let $\eta$ be the unique ergodic measure on the shift space $Z$ consisting of two points $((-1)^{i})_i$ and $((-1)^{i+1})_i$. The image $\lambda$ of $\mu \otimes \eta$ under the map
$$(x, z) \mapsto (x, x+z, x+2z, x+3z, x+4z) \mod5$$
is a degree joining over $\pi\mu$. ($\lambda$ is ergodic because $\mu\otimes\eta$ is, by the previous lemma.) By conditioning on $x$, we can verify that the second margin and the last margin are the same and we denote it by $\mu'$. Also, the third margin and the fourth margin are the same and we denote it by $\mu''$. The measures $\mu, \mu', \mu''$ are all ergodic lifts of $\pi\mu$.
Typically but not always, $\mu, \mu', \mu''$ will be three distinct measures and their multiplicities will be 1, 2, 2.
Whenever $\nu$ is an ergodic measure on $Y$ such that 2 is not its factor, the number of its ergodic lifts is at most 3.
\end{example}
\begin{example}
  Let $\pi: X\to Y$ and $(Z, \eta)$ be from the previous example, but this time we suppose $\mu$ is an ergodic measure on $X$ such that 2 is its factor. Then there is a shift-commuting measurable function $F: X \to Z$ such that $F\mu = \eta$. The image of $\mu$ under the map
$$x \mapsto (x, x+ F(x), x+ 2F(x), x+ 3F(x), x+ 4F(x)) \mod5$$
is a degree joining over $\pi\mu$. Typically but not always, the margins will be all different and their multiplicities will be 1.
\end{example}


\section{Class degree}

Given a 1-block factor code $\pi:X\to Y$ on a 1-step SFT $X$, the notion of class degree and transition classes were introduced in \cite{all2013classdegrelmaxent}. Given a point $y \in Y$, the fiber $\pi^{-1}(y)$ is divided into finitely many equivalence classes $[x]$ called transition classes over $y$. The equivalence relation is as follows: the two points $x, x' \in \pi^{-1}(y)$ are equivalent if there is a transition from $x$ to $x'$ via finite $X$-words while keeping the same image $y$ infinitely many times to the right and vice versa from $x'$ to $x$. The equivalence relation is denoted by $x \sim x'$ and the equivalence class is denoted by $[x]$.


\begin{theorem}[Consequence of Theorem 4.22 in \cite{all2013classdegrelmaxent}]
  Let $\pi:X \to Y$ be a factor code on an SFT $X$ and $\nu$ an ergodic measure on $Y$. Then $\nu$-almost every point of $Y$ (in fact, every recurrent point that visits every neighborhood of every point in the topological support of $\nu$) has the same finite number of transition classes over it. 
\end{theorem}
We will call this number the class degree of $\nu$ and denote it by $c_{\pi,\nu}$ or $c_{\nu}$.
If $\nu$ has full support, then the number $c_{\pi,\nu}$ does not depend on $\nu$ and is the class degree of $\pi$, the number of transition classes over any doubly transitive point in $Y$ as defined in \cite{all2013classdegrelmaxent}.

Given two invariant measures $\mu, \mu'$ on $X$, we have $\pi\mu = \pi\mu'$ if and only if there is a relative joining of the two. We will denote by $C_\pi$ the set of $(x,x') \in X^2$ such that $\pi(x)=\pi(x')$ and $x \sim x'$.

Given two \emph{ergodic} measures $\mu, \mu'$ on $X$, we will say they are \emph{class parallel} if there is a relative joining $\lambda$ of $\mu$ and $\mu'$ such that $\lambda(C_\pi) =1$.
\begin{theorem}
  Let $\pi:X \to Y$ be a factor code on an SFT $X$ and $\nu$ an ergodic measure on $Y$. Let $\mu, \mu'$ be two ergodic lifts of $\nu$. Then the following are equivalent.
  \begin{enumerate}
  \item $(\mu \otimes_\nu \mu')(C_\pi) > 0$.
  \item There is a relative joining $\lambda$ of the two such that $\lambda(C_\pi) > 0$. 
  \item There is an ergodic relative joining $\lambda$ of the two such that $\lambda(C_\pi)=1$.
  \item For $\nu$-a.e. $y \in Y$, the transition classes $[x]$ over $y$ for which $\mu_y([x]) >0$ are precisely the transition classes for which $\mu'_y([x])>0$.
  \item For $\nu$-a.e. $y \in Y$, there is a transition class $[x]$ over $y$ such that $\mu_y([x]) > 0$ and $\mu'_y([x]) >0$.
  \end{enumerate}
\end{theorem}
\begin{proof}
  The implication $(1) \implies (2)$ is trivial. To see $(2) \implies (3)$, use ergodic decomposition and the invariance of $C_\pi$.

  To see $(3)\implies(4)$, first verify that $\lambda_y$ (disintegration of $\lambda$ by $\pi\circ p_1$) is a coupling of $\mu_y$ and $\mu'_y$ and that $\lambda_y(C_\pi)=1$ for $\nu$-a.e. $y$. Fix $y \in Y$ to be one of such points and let $[x_1],[x_2],\dots,[x_k]$ be all transition classes over $y$. Then $\lambda_y$ is supported on $\cup_{i=1}^k ([x_i]\times[x_i])$. Let $J$ be the set of $i$ for which $\lambda_y([x_i]\times[x_i]) > 0$. Then the transition classes corresponding to $J$ are precisely the transition classes to which the measure $p_1\lambda_y = \mu_y$ gives positive measure. Similarly, this also holds for $p_2\lambda_y = \mu'_y$.

  The implication $(4)\implies(5)$ is trivial.

  To see $(5)\implies(1)$, first observe that $(\mu_y \otimes \mu'_y)(C_\pi) > 0$ holds a.e. and then integrate over $Y$.
\end{proof}
We have listed equivalent conditions for being class parallel. For convenience, we will also list equivalent conditions for its negation.
\begin{theorem}
  Let $\pi:X \to Y$ be a factor code on an SFT $X$ and $\nu$ an ergodic measure on $Y$. Let $\mu, \mu'$ be two ergodic lifts of $\nu$. Then the following are equivalent.
  \begin{enumerate}
  \item $(\mu \otimes_\nu \mu')(C_\pi) = 0$.
  \item For any relative joining $\lambda$ of $\mu$ and $\mu'$ we have $\lambda(C_\pi)=0$.
  \item For $\nu$-a.e. $y \in Y$, there is no transition class $[x]$ over $y$ such that $\mu_y([x]) > 0$ and $\mu'_y([x]) >0$.
  \end{enumerate}
\end{theorem}
\begin{proof}
  The implications $(1)\implies(2)\implies(1)$ follows from the previous theorem. The implications $(1)\implies(3)\implies(1)$ are trivial.
\end{proof}

Being class parallel is an equivalence relation 
and the third property in the previous theorem above ensures that this equivalence relation partitions the set of ergodic lifts of $\nu$ into at most $c_{\pi,\nu}$ equivalence classes.

We will say an $n$-fold relative joining $\lambda$ is \emph{class separating} if for $\lambda$-a.e. $(x^{(1)}, x^{(2)}, \dots, x^{(n)})$, the points $x^{(1)}, x^{(2)}, \dots, x^{(n)}$ are in $n$ different transition classes over $\pi x^{(1)} = \cdots = \pi x^{(n)}$.
\begin{theorem}\label{thm:class-degree-joining-exists}
  Let $\pi:X \to Y$ be a factor code on an SFT $X$ and $\nu$ an ergodic measure on $Y$. Then there exists an ergodic $c_{\pi,\nu}$-fold class separating relative joining over $\nu$. We will call such a joining a \emph{class degree joining} over $\nu$ with respect to $\pi$.
\end{theorem}
For $\nu$-a.e. $y \in Y$, we can give an ordering on the $c$ transition classes over $y$ or more simply, select one transition class over $y$, but in general we cannot do it in a measurable, invariant way. A class degree joining is a generalization and substitute for such selection functions. In order to prove its existence, we first build an auxiliary shift space which is an infinite-to-one analogue of the shift space of mutually separated points used in the proof of existence of a degree joining for finite-to-one factor codes.

Given a 1-block factor code $\pi:X\to Y$ on a 1-step SFT $X$, a \emph{bi-transition} between two $X$-words $U, U'$ (of length $L$) with same image $\pi(U) = \pi(U') = V$ is a pair of $X$-words $W, W'$ such that
\begin{align*}
  \pi(W) = \pi(W') = V \,, \\
  W_0 = U_0 \,,\quad W_{L-1} = U'_{L-1} \,, \\
  W'_0 = U'_0 \,,\quad W'_{L-1} = U_{L-1} \,.
\end{align*}

Let $\overline X^n_\pi$ be the set of points $(x^{(1)}, \dots, x^{(n)}) \in X^n$ such that $\pi(x^{(1)}) = \dots = \pi(x^{(n)})$ and that for each $-\infty < i < j < \infty$ and each $1 \le k < l \le n$, there is no bi-transition between the words $x^{(k)}_{[i,j]}$ and $x^{(l)}_{[i,j]}$.
This set is either empty or a shift space and we will call this the $n$-fold \emph{bi-transition forbidding product} of $X$ with respect to $\pi$.
Note that for each $(x^{(1)}, \dots, x^{(n)}) \in \overline X^n_\pi$, the $n$ points $x^{(1)}, \dots, x^{(n)}$ are in $n$ different transition classes over $\pi(x^{(1)}) = \dots = \pi(x^{(n)})$. 

\begin{lemma}
  Let $\pi:X \to Y$ be a 1-block factor code on a 1-block SFT $X$. Given a recurrent point $y \in Y$ with at least $c$ transition classes over it, there is a point in $\overline X^c_\pi$ that maps to it.
\end{lemma}
\begin{proof}
  Let $y \in Y$ be a recurrent point and fix $c$ points $x^{(1)}, \dots, x^{(c)}$ from $c$ different transition classes over $y$. There is $m$ such that for each $n \in \mathbb N$, there is no bi-transition between any two of the $c$ words $x^{(1)}_{[m,m+n]}, \dots, x^{(c)}_{[m,m+n]}$, because otherwise by the pigeon hole's principle, we would be able to pick two points $x^{(i)}, x^{(j)}$ among the $c$ points such that bi-transitions occur infinitely many times to the right between the two points, which contradicts how we chose the $c$ points. We may assume $m=0$ without loss of generality, so that for each $n \in \mathbb N$, there is no bi-transition between any two of the $c$ words $x^{(1)}_{[0,n]}, \dots, x^{(c)}_{[0,n]}$.

Since $y$ is recurrent, there is a sequence $n_k \nearrow \infty$ such that $\sigma^{n_k}(y) \to y$ as $k \to \infty$. Since $X^c$ is compact, we may assume, by passing to a subsequence, that $\sigma^{n_k}(x^{(1)}, \dots, x^{(c)})$ converges in $X^c$. Let $(x'^{(1)}, \dots, x'^{(c)})$ be the limit. It is then easy to verify that for each $1 \le k \le c$, $\pi(x'^{(k)}) = y$ and that $(x'^{(1)}, \dots, x'^{(c)}) \in \overline X^c_\pi$.
\end{proof}


\begin{lemma}
  Let $\pi:X \to Y$ be a 1-block factor code on a 1-block SFT $X$ and $\nu$ an ergodic measure on $Y$ with $c = c_{\pi,\nu}$. Then for each $y \in \supp(\nu)$, there is a point in $\overline X^c_\pi$ that maps to it. In other words, the image of $\overline X^c_\pi$ under the obvious map $\overline X^c_\pi \to Y$ contains the topological support of $\nu$.
\end{lemma}
\begin{proof}
Let $Y_0$ be the image of $\overline X^c_\pi$ under the obvious map $\overline X^c_\pi \to Y$. Then $Y_0$ is a shift space.
Suppose $y \in \supp(\nu) \setminus Y_0$.
Since $Y_0$ is closed, there is $n$ such that the cylinder $[y_{[-n,n]}]_{-n}$ and $Y_0$ are disjoint.
Since $Y_0$ is shift-invariant, it follows that the word $W:= y_{[-n,n]}$ can never occur in points of $Y_0$.
But since $y \in \supp(\nu)$, we have $\nu(W) > 0$. Therefore, for $\nu$-a.e. $y' \in Y$, $W$ occurs in $y'$ and the point $y'$ is recurrent and there are $c$ transition classes over $y'$. Fix $y'$ to be one of such points. By the previous lemma, $y' \in Y_0$, but this contradicts that $W$ can never occur in points of $Y_0$.
\end{proof}

\begin{proof}[Proof of Theorem~\ref{thm:class-degree-joining-exists}]
  We may assume $\pi:X\to Y$ is a 1-block factor code on a 1-block SFT $X$.

  Consider the factor code $\overline\pi: \overline X^c_\pi \to Y$ induced by $\pi$. Since the image contains $\supp(\nu)$, it can be lifted to an ergodic measure $\lambda$ on $\overline X^c_\pi$. It is then easy to verify that $\lambda$ is a class degree joining over $\nu$ with respect to $\pi$.
\end{proof}

Given two $n$-fold relative joinings $\lambda, \lambda'$ over $\nu$, we will say they are class parallel if there is an $2n$-fold relative joining $\lambda''$ over $\nu$ such that $p_{[1,n]}\lambda'' = \lambda$ and $p_{[n+1,2n]}\lambda'' = \lambda'$ and that for $\lambda''$-a.e. $(x_1,\dots,x_n,x'_1,\dots,x'_n)$ in $X^n \times X^n$, for each $1\le i \le n$, the two points $x_i, x'_i$ are in the same transition class. Such $\lambda''$ will be called a joining that realizes the class parallel relation between $\lambda$ and $\lambda'$. 

\begin{theorem}\label{thm:class-degree-joining-universal}
  Let $\pi:X \to Y$ be a factor code on an SFT $X$ and $\nu$ an ergodic measure on $Y$ with $c = c_{\pi,\nu}$. Let $\lambda$ be a class degree joining over $\nu$. Let $\lambda'$ be an $n$-fold ergodic relative joining over $\nu$. Then there is a function $f: \{1,\dots, n\} \to \{1,\dots,c\}$ such that $p_f\lambda$ is class parallel to $\lambda'$ where $p_f: X^c \to X^n$ is the obvious map induced by $f$. Furthermore, if $\lambda'$ is class separating, we can choose $f$ to be injective.
\end{theorem}
\begin{proof}
    There is an ergodic relative joining, say $\lambda''$, of $\lambda$ and $\lambda'$ over $\nu$ in the sense of Lemma~\ref{lem:ergodic-rel-join-exists}. $\lambda''$ is a measure on $X^{c} \times X^{n}$. For $\lambda''$-a.e. $(x_1, \dots, x_c, x'_1, \dots, x'_n)$, we have that the points $x_1, \dots, x_c$ are in $c$ distinct transition classes over $\pi(x_1)$ and that $\pi(x_1)$ only has $c$ transition classes and that $x'_1,\dots,x'_n$ are pre-images of $\pi(x_1)$.  Therefore in particular, the point $x'_1$ belongs to one and only transition class among $[x_1], \dots, [x_c]$. Therefore there is a measurable function $g: X^{c} \times X^{n} \to \{1,\dots, c\}$ such that $x'_1$ and $x_{g(x_1, \dots, x_c, x'_1, \dots, x'_n)}$ are in same transition class for $\lambda''$-almost all $(x_1, \dots, x_c, x'_1, \dots, x'_n)$.

Since $g$ is invariant and $\lambda''$ is ergodic, $g$ must be constant $\lambda''$-a.e.. Define $f(1)$ to be this constant. Define $f(2), \dots, f(n)$ similarly. Then we have
\begin{align*}
  x'_1 &\sim x_{f(1)}\\
  x'_2 &\sim x_{f(2)}\\
  &\dots\\
  x'_n &\sim x_{f(n)}
\end{align*}
This means that the image of $\lambda''$ under the map
$$(x_1, \dots, x_c, x'_1, \dots, x'_n) \mapsto (x'_1,\dots,x'_n, x_{f(1)},\cdots,x_{f(n)})$$
realizes the class parallel relation between $\lambda'$ and $p_f \lambda$.

Furthermore, if $\lambda'$ is class separating, then in particular, $x'_1$ and $x'_2$ are in different transition classes, and hence $f(1) \neq f(2)$ and the same can be said for any two indexes $1\le i < j \le n$.
\end{proof}

\begin{theorem}
  Let $\pi:X \to Y$ be a factor code on an SFT $X$ and $\nu$ an ergodic measure on $Y$ with $c = c_{\pi,\nu}$. Let $\lambda, \lambda'$ be class degree joinings over $\nu$. Then there is a permutation $f: \{1,\dots, c\} \to \{1,\dots,c\}$ such that $p_f\lambda$ is class parallel to $\lambda'$.
\end{theorem}
\begin{proof}
    There is a function $f: \{1,\dots, c\} \to \{1,\dots,c\}$ such that $\lambda'$ is class parallel to $p_f \lambda$ and we can choose $f$ to be injective. Since the domain and the codomain of $f$ have the same size, $f$ is also surjective.
\end{proof}
Therefore class degree joining over $\nu$ is unique up to permutation and class parallel change.

\begin{theorem}
  Let $\pi:X \to Y$ be a factor code on an SFT $X$ and $\nu$ an ergodic measure on $Y$ with $c = c_{\pi,\nu}$. Let $\lambda$ be a class degree joining over $\nu$. Given any ergodic lift $\mu$ of $\nu$, it is class parallel to at least one of the margins of $\lambda$.
\end{theorem}
\begin{proof}
  $\mu$ is a 1-fold ergodic relative joining over $\nu$ and so Theorem~\ref{thm:class-degree-joining-universal} applies to it and therefore $\mu$ is class parallel to one of the margins of $\lambda$.
\end{proof}
Therefore if $[\mu_1], \dots, [\mu_k]$ are all the class parallel equivalence classes for the ergodic lifts of $\nu$, then we have
$$\{[p_i\lambda]: 1\le i \le c\} = \{[\mu_i]: 1\le i \le k\}$$
\begin{theorem}\label{thm:margins}
  Let $\pi:X \to Y$ be a factor code on an SFT $X$ and $\nu$ an ergodic measure on $Y$ with $c = c_{\pi,\nu}$. Let $\lambda$ and $\lambda'$ be class degree joinings over $\nu$. Then after applying a permutation to $\lambda'$, we have that the $i$'th margin of $\lambda$ is class parallel to the $i$'th margin of $\lambda'$ for each $1 \le i \le c$. Also, if $\mu_1, \dots, \mu_c$ are ergodic lifts of $\nu$ such that $\mu_i$ is class parallel to the $i$'th margin of $\lambda$ for each $i$, then there is a class degree joining over $\nu$ whose $i$'th margin is $\mu_i$ for each $i$.
\end{theorem}
\begin{proof}
  Apply a permutation to $\lambda'$ to make it class parallel to $\lambda$. Then let $\lambda''$ be a joining that realizes this class parallel relation. The image of $\lambda''$ under the map
  $$(x_1,\dots,x_c, x'_1,\dots,x'_c) \mapsto (x_i,x'_i)$$
  is a joining that realizes the class parallel relation between $p_i \lambda$ and $p_i \lambda'$.

  For the latter part, by induction, it is sufficient to construct a class degree joining $\overline \lambda$ over $\nu$ such that $\mu_1 = p_1 \overline\lambda$ and $p_i \lambda = p_i \overline\lambda$ for the rest $2 \le i \le c$. Let $\rho$ be an ergodic joining that realizes the class parallel relation between $\mu_1$ and $p_1 \lambda$. Let $\rho'$ be an ergodic relative joining of $\rho$ and $\lambda$ in the sense of Theorem~\ref{lem:ergodic-rel-join-exists}. 
For $\rho'$-a.e. $(x',x'', x_1,\dots,x_c)$, we have that $x' \sim x'' = x_1$ and that $x_1, \dots, x_c$ are in different transition classes over $\pi(x_1) = \dots = \pi(x_c)$. This means that the image of $\rho'$ under the map
$$(x',x'', x_1,\dots,x_c) \mapsto (x', x_2,\dots,x_c)$$
is class separating. This image is the desired class degree joining with the right margins.
\end{proof}

\section{Class multiplicity}

\begin{definition}
  Let $\pi:X\to Y$ be a factor code on an SFT $X$. Let $\mu$ be an ergodic measure on $X$. Then the \emph{class multiplicity}, denoted $m_\pi(\mu)$, of $\mu$ with respect to $\pi$ is defined by
  $$m_\pi(\mu) := \#\{i : 1\le i \le c,\ p_i \lambda \sim \mu\}$$
  where $\lambda$ is a class degree joining over $\pi\mu$ and $\sim$ denotes the class parallel relation.
\end{definition}
Since class degree joining is unique up to permutation and class parallel relation, the notion of class multiplicity is well defined.
Also, if $\mu$ and $\mu'$ are two ergodic measures on $X$ and if they are class parallel, then $m_\pi(\mu) = m_\pi(\mu')$.

\begin{theorem}
  Let $\pi:X \to Y$ be a factor code on an SFT $X$. If $[\mu_1], \dots, [\mu_k]$ are all the equivalence classes for the ergodic lifts of an ergodic measure $\nu$ on $Y$, then
  $$c_{\pi,\nu}= \sum_{i=1}^k m_\pi([\mu_i])$$
\end{theorem}

\begin{lemma}
  Let $Q$ be a measure on $\{1,\dots,c\}^c$ such that
  $$Q(\{(x_1, \dots,x_c): x_i \neq x_j \text{ for all } 1\le i< j\le c \})=1$$
  and that for each $1\le i<j \le c$, the two margins $p_iQ$ and $p_jQ$ are either equal or mutually singular.
  Let $P_1,\dots,P_k$ be all distinct margins of $Q$. ($k \le c$) Define
  $$I_1 = \{i: P_1(i) > 0\,, 1\le i \le c \}$$
  $$J_1 = \{i: p_iQ = P_1\,, 1\le i\le c\}$$
  Then $I_1, \dots, I_k$ form a partition of $\{1,\dots,c\}$ and the size of $I_i$ is the same as that of $J_i$ and each $P_i$ is the uniform distribution on $I_i$. For each atom $(x_1, \dots, x_c)$ of $Q$, we have $\{x_i: i \in J_1\} = I_1$ and likewise for $J_2, \dots, J_k$.
\end{lemma}
\begin{proof}
  Since distinct margins are mutually singular, we have
  $$\#I_1 + \dots + \#I_k \le c$$
  Let $(x_1, \dots, x_c)$ be an atom of $Q$. Then $x_i$, for $i \in J_1$, are $\#J_1$ distinct atoms of $P_1$. Therefore $\#I_1 \ge \#J_1$. Similarly, $\#I_i \ge \#J_i$ for each $i$, but since
  $$\#J_1 + \dots + \#J_k = c$$
  we must have $\#I_i = \#J_i$ and
  $$\#I_1 + \dots + \#I_k = c$$
  And therefore $I_1, \dots, I_k$ form a partition of $\{1,\dots,c\}$.
  Since $\#I_i = \#J_i$, the points $x_i$, for $i \in J_1$, are all the atoms of $P_1$.

  Let $U$ be the uniform distribution on $J_1$. Define $F: \{1,\dots,c\}^c \times J_1 \to \{1,\dots,c\}$ by
  $$(x_1,\dots,x_c, i)\mapsto x_i$$
  Then the measure $F(Q \otimes U)$ is equal to $P_1$ by conditioning on $U$, but $F(Q \otimes U)$ is also equal to the uniform distribution on $I_1$ by conditioning on $Q$.
\end{proof}

\begin{theorem}
Let $\pi:X \to Y$ be a factor code on an SFT $X$. Let $[\mu_1],\dots,[\mu_k]$ be all the class parallel equivalence classes for the ergodic lifts of an ergodic measure $\nu$ on $Y$. Let $m_1,\dots, m_k$ be their class multiplicities and let $\mu_{1,y}, \dots \mu_{c,y}$ be their disintegrations over $Y$. Then for $\nu$-a.e. $y\in Y$, the fiber $\pi^{-1}(y)$ consists of $c$ transition classes $C_{i,j}(y)$, for $1\le i \le k$ and $1\le j \le m_i$, such that
\[
\mu_{1,y}(C_{i,j}(y)) = 
\begin{cases}
  \frac1{m_1} & \text{if } i=1\\
  0 & \text{if } i \neq 1
\end{cases}
\]
and likewise for $\mu_2,\dots,\mu_k$.
\end{theorem}
\begin{proof}
  Let $\{\lambda_y\}_{y\in Y}$ be a disintegration of a class degree joining over $\nu$. We may assume that the first $m_1$ margins of $\lambda$ are $\mu_1$ and the next $m_2$ margins are $\mu_2$ and so on. For $\nu$-a.e. $y\in Y$, the fiber $\pi^{-1}(y)$ consists of $c$ transition classes and the first $m_1$ margins of $\lambda_y$ are $\mu_{1,y}$ and the next $m_2$ margins of $\lambda_y$ are $\mu_{2,y}$ and so on, and the measures $\mu_{1,y}, \dots, \mu_{c,y}$ are all supported on $\pi^{-1}(y)$ and each transition class over $y$ has positive measure for at most one of the measures $\mu_{1,y}, \dots, \mu_{c,y}$, and $\lambda_y$ is class separating. Fix $y$ to be such a point.

  Give an ordering to the $c$ transition classes over $y$ and let $F: \pi^{-1}(y) \to \{1,\dots,c\}$ be the measurable map induced by the ordering. Let $F': (\pi^{-1}(y))^c \to \{1,\dots,c\}^c$ be the map induced by $F$. Let $Q = F'(\lambda_y)$. Then $Q$ satisfies the conditions of the previous lemma and the first $m_1$ margins of $Q$ are $F(\mu_1)$ and the next $m_2$ margins of $Q$ are $F(\mu_2)$ and so on.

By the lemma, there are exactly $m_1$ transition classes that $\mu_1$ gives positive measure to and it gives $\frac1{m_1}$ to each of the $m_1$ transition classes. Let $C_{1,j}(y)$, for $1\le j \le m_1$, be those transition classes. Define the rests $C_{i,j}(y)$ similarly.
\end{proof}

\begin{theorem}
  Let $\pi:X \to Y$ be a factor code on an SFT $X$. Let $\mu$ be an ergodic measure on $X$ and $m$ its class multiplicity. Let $\{\mu_y\}_{y \in Y}$ be the disintegration of $\mu$ over $Y$. Then
  \begin{enumerate}
  \item For $\pi\mu$-a.e. $y$, the measure $\mu_y$ assigns $\frac1m$ to $m$ transition classes over $y$ and zero to the rest.
  \item $m$ is the maximum number such that there is an $m$-fold class separating relative joining of $\mu, \dots, \mu$.
  \item $(\mu \otimes_{\pi\mu} \mu) \{(x, x'): x \sim x'\} = \frac1m$
  \end{enumerate}
\end{theorem}
\begin{proof}
  Let $\nu = \pi\mu$. The third property follows from the first property which in turn follows from the previous theorem, so we only need to prove the second property.

Let $\lambda'$ be a class degree joining over $\nu$. We may assume $\mu = p_i \lambda'$ for $1 \le i \le m$, by Theorem~\ref{thm:margins}. Let $\lambda$ be the projection of $\lambda'$ to the first $m$ coordinates. Then it is an $m$-fold class separating relative joining of $\mu, \dots, \mu$.

Suppose $\lambda''$ is an $(m+1)$-fold class separating relative joining of $\mu, \dots,\mu$. We may assume that it is ergodic. Then by Theorem~\ref{thm:class-degree-joining-universal}, we may assume it is a projection of a class degree joining $\lambda'''$ to the first $m+1$ coordinates. Since the first $m+1$ margins of $\lambda'''$ is $\mu$, we can conclude that the class multiplicity of $\mu$ is at least $m+1$, which contradicts our assumption.
\end{proof}

\section{Thermodynamic formalism}

\begin{lemma}
  Let $\lambda$ be a class degree joining over an ergodic measure $\nu$ on $Y$ and let $\mu$ a (possibly non-ergodic) invariant measure that is class parallel to the first margin of $\lambda$. Then there is a $(c_\nu+1)$-fold relative joining $\lambda'$ such that its projection to the first $c$ coordinates is $\lambda$ and its last margin is $\mu$ and that $x_1 \sim x$ holds for $\lambda'$-a.e $(x_1,\dots,x_c, x)$.
\end{lemma}
\begin{proof}
  Let $\rho$ be a joining that realizes the class parallel relation between $p_1\lambda$ and $\mu$. Let $\rho'$ be any relative joining of $\lambda$ and $\rho$ under the factor maps $p_1: X^c \to X$ and $p_1: X^2 \to X$ in the sense of Theorem~\ref{lem:ergodic-rel-join-exists}. Then for $\rho'$-a.e. $(x_1,\dots,x_c, x'_1,x'_2)$ we have $x_1 = x'_1 \sim x'_2$. Therefore the image of $\rho'$ under the map $$(x_1,\dots,x_c, x'_1,x'_2) \mapsto (x_1,\dots,x_c, x'_2)$$ is a joining with the desired properties.
\end{proof}

\begin{theorem}
  Given an ergodic measure $\mu$ on $X$, the set $S_\mu$ of all (possibly non-ergodic) invariant measures class parallel to it is non-empty, convex and compact.
\end{theorem}
\begin{proof}
  It is non-empty because $\mu$ is in it.

  Suppose $\mu_1, \mu_2 \in S_\mu$. Let $\mu_3$ be a convex combination of $\mu_1, \mu_2$ and write $\mu_3 = \alpha_1 \mu_1 + \alpha_2 \mu_2$. Let $\lambda_1, \lambda_2$ be joinings that realize class parallel relation from $\mu$ to $\mu_1, \mu_2$. Then $\alpha_1\lambda_1 + \alpha_2\lambda_2$ is a joining that realizes the class parallel relation between $\mu$ and $\mu_3$. Therefore $S_\mu$ is convex.

  Suppose $(\mu_i)_i$ is a sequence in $S_\mu$ that converges to a measure $\mu_*$. Let $\nu = \pi\mu$. Then we also have $\nu = \pi\mu_i = \pi\mu_*$ and $\mu_*$ is invariant.

  Let $c$ be the class degree of $\nu$. Let $\lambda$ be a class degree joining over $\nu$. We may assume that its first margin is $\mu$. By the previous lemma, there is a $(c+1)$-fold relative joining $\lambda_i$ such that its projection to the first $c$ coordinates is $\lambda$ and its last margin is $\mu_i$ and that $x_1 \sim x$ holds for $\lambda_i$-a.e $(x_1,\dots,x_c, x)$.

  Let $U_\ell \subset X^{c+1}$ be the set of $(x_1,\dots,x_c, x)$ for which there is $-\infty < i < j < \infty$ such that the restrictions of $x$ and $x_\ell$ to the region $[i,j]$ have the same image and that there is a bi-transition between the two. Let $U$ be the union of the $c-1$ open sets $U_2, \dots, U_c$.

  If $\lambda_i(U_2) > 0$, then Poincare's recurrence theorem says that for $\lambda_i$-a.e. points in $U_2$ we get $x \sim x_2$, but since $x \sim x_1$, this would imply $x_1 \sim x_2$ which contradicts the fact that $\lambda$ is class separating. Therefore $\lambda_i(U_2) = 0$ and in fact $\lambda_i(U) =0$.

  By passing to a subsequence, we may assume that $\lambda_i$ converges to some measure $\lambda_*$. Then $\lambda_*$ is a relative joining over $\nu$ such that its projection to the first $c$ coordinates is $\lambda$ and its last margin is $\mu_*$.
  Since $U$ is open, we get $\lambda_*(U) = 0$. Since $\lambda$ is class separating, this forces $\lambda_*(x_1 \sim x) = 1$. This means that the image of $\lambda_*$ under the map
  $$(x_1,\dots,x_c, x) \mapsto (x_1,x)$$
  is a joining that realizes the class parallel relation between $\mu$ and $\mu_*$. Therefore $\mu_* \in S_\mu$ and $S_\mu$ is compact.
\end{proof}

\begin{theorem}
  Given an ergodic measure $\mu$ on $X$ and a continuous function $f$ on $X$, let $S_{\mu,f}$ be the set of measures $\mu'' \in S_\mu$ for which $$h(\mu'') + \int f d\mu'' \ge h(\mu') + \int f d\mu'$$ for all $\mu' \in S_\mu$. Then $S_{\mu,f}$ is non-empty, convex and compact.
\end{theorem}
\begin{proof}
  Since the map $$\mu'' \mapsto h(\mu'') + \int f d\mu''$$
  is upper semi-continuous and $S_\mu$ compact, the set $S_{\mu,f}$ is non-empty. It is convex because the map is affine. It is compact because the map is upper semi-continuous.
\end{proof}

We suspect that the set $S_{\mu,f}$ consists of just one measure whenever $f$ is regular enough. The result $\#S_{\mu,f} \ge 1$ (positive answer to Question B' for all continuous $f$) is established here by relying on class degree joining. The dual result $\#S_{\mu,f} \le 1$ (positive answer to Question A') for $f$ with summable variation is a consequence of a lemma in \cite{yoo2014releqclass}.

We will say that an ergodic measure $\mu$ on $X$ is \emph{class maximal} if $\mu \in S_{\mu,0}$. An ergodic lift $\mu$ of an ergodic measure $\nu$ on $Y$ is class maximal if and only if it is maximal within its equivalence class $[\mu]$ (the set of all ergodic lifts of $\nu$ parallel to $\mu$), by the following lemma. In general, a class maximal measure is not a relatively maximal measure. 

\begin{lemma}
  Given an ergodic measure $\mu$ on $X$ and $\mu'$ an invariant measure on $X$ with ergodic decomposition $\mu' = \int \mu'' d\rho(\mu'')$. Then $\pi\mu = \pi\mu'$ if and only if $\pi\mu = \pi\mu''$ for $\rho$-a.e. $\mu''$. And $\mu \sim \mu'$ if and only if $\mu \sim \mu''$ for $\rho$-a.e. $\mu''$.
\end{lemma}
\begin{proof}
  $$\pi\mu' = \int \pi\mu'' d\rho(\mu'')$$
  If $\pi\mu = \pi\mu'$, then the above equation is an ergodic decomposition of an ergodic measure, and hence $\pi\mu = \pi\mu' = \pi\mu''$ for $\rho$-a.e. $\mu''$. Its converse follows similarly.

  Suppose $\mu \sim \mu'$. Let $\lambda'$ be a joining that realizes the class parallel relation between $\mu$ and $\mu'$. Since $\lambda'(C_\pi)=1$, if $\lambda' = \int \lambda'' d\tau(\lambda'')$ is an ergodic decomposition of $\lambda'$, then $\lambda''(C_\pi) = 1$ holds for $\tau$-a.e. $\lambda''$. By projecting to the second coordinate, we get $\mu' = \int p_2 \lambda'' d\tau(\lambda'')$ which is an ergodic decomposition of $\mu'$. The joining $\lambda''$ realizes the class parallel relation between $\mu$ and $p_2 \lambda''$ and hence $\mu \sim p_2 \lambda''$ for $\tau$-a.e. $\lambda''$. Therefore $\mu \sim \mu''$ for $\rho$-a.e. $\mu''$.

  Conversely, suppose $\mu \sim \mu''$ for $\rho$-a.e. $\mu''$ and let $\lambda''_{\mu''}$ be a joining that realizes this relation for $\rho$-a.e. $\mu''$. Then $\int \lambda''_{\mu''} d\rho(\mu'')$ realizes the class parallel relation between $\mu$ and $\mu'$.
\end{proof}

\begin{theorem}
  Let $\mu_1,\dots,\mu_k$ be all the class maximal measures among the ergodic lifts of an ergodic measure $\nu$ on $Y$. Let $c = c_{\nu}$ and let $m_1,\dots,m_k$ be class multiplicities of the $k$ measures. Then
  $$c = \sum_{i=1}^k m_i$$
\end{theorem}
\begin{proof}
  Each $\mu_i$ is the unique class maximal measure in each equivalence class. The conclusion follows.
\end{proof}
A similar result holds for measures maximizing $h(\mu) + \int f d\mu$ in their equivalence classes, as long as $f$ is a function for which $\#S_{\mu,f} \le 1$ is proved.







  


%

\bibliographystyle{amsplain}
\bibliography{my-reference}
\end{document}